\documentclass{amsart}

\usepackage{amsmath,amsthm, amsfonts, amssymb, mathrsfs}
\usepackage[hidelinks]{hyperref}
\usepackage{cite}
\usepackage[capitalize]{cleveref} 
\usepackage{tikz, tikz-cd} 
\usetikzlibrary{positioning}
\usetikzlibrary{decorations.markings}
\usepackage{dynkin-diagrams}
\usepackage{parskip, verbatim}
\usepackage{microtype,fixltx2e,lmodern} 
\usepackage[utf8]{inputenc} 
\usepackage[T1]{fontenc} 
\usepackage{enumerate,comment,braket,xspace,csquotes} 
\usepackage[all,cmtip]{xy} 

\usepackage{multirow}

\usepackage{xcolor}
\hypersetup{
	colorlinks,
	linkcolor={red!50!black},
	citecolor={blue!50!black},
	urlcolor={blue!80!black}
}

\newcommand{\rH}{\mathrm H}

\newcommand{\fS}{\mathfrak S}

\newcommand{\cF}{\mathcal F}

\newcommand{\cL}{\mathcal L}
\newcommand{\cM}{\mathcal M}

\newcommand{\cO}{\mathcal O}

\newcommand{\cV}{\mathcal V}

\DeclareFontFamily{U}{BOONDOX-calo}{\skewchar\font=45 }
\DeclareFontShape{U}{BOONDOX-calo}{m}{n}{<-> s*[1.05] BOONDOX-r-calo}{}
\DeclareFontShape{U}{BOONDOX-calo}{b}{n}{<-> s*[1.05] BOONDOX-b-calo}{}
\DeclareMathAlphabet{\mathcalboondox}{U}{BOONDOX-calo}{m}{n}

\newcommand{\bZ}{\mathbb{Z}}							  				
						
\newcommand{\bT}{\mathbb{T}}						
						
\newcommand{\bB}{\mathbb{B}}						
\newcommand{\bX}{\mathbb{X}}					

\newcommand{\bQ}{\mathbb{Q}}
\newcommand{\bC}{\mathbb{C}}

\newcommand{\bG}{\mathbb{G}}
\newcommand{\bk}{\mathbf{k}}

\makeatletter
\let\save@mathaccent\mathaccent
\newcommand*\if@single[3]{%
	\setbox0\hbox{${\mathaccent"0362{#1}}^H$}%
	\setbox2\hbox{${\mathaccent"0362{\kern0pt#1}}^H$}%
	\ifdim\ht0=\ht2 #3\else #2\fi
}
\newcommand*\rel@kern[1]{\kern#1\dimexpr\macc@kerna}
\newcommand*\widebar[1]{\@ifnextchar^{{\wide@bar{#1}{0}}}{\wide@bar{#1}{1}}}
\newcommand*\wide@bar[2]{\if@single{#1}{\wide@bar@{#1}{#2}{1}}{\wide@bar@{#1}{#2}{2}}}
\newcommand*\wide@bar@[3]{%
	\begingroup
	\def\mathaccent##1##2{%
		\let\mathaccent\save@mathaccent
		\if#32 \let\macc@nucleus\first@char \fi
		\setbox\z@\hbox{$\macc@style{\macc@nucleus}_{}$}%
		\setbox\tw@\hbox{$\macc@style{\macc@nucleus}{}_{}$}%
		\dimen@\wd\tw@
		\advance\dimen@-\wd\z@
		\divide\dimen@ 3
		\@tempdima\wd\tw@
		\advance\@tempdima-\scriptspace
		\divide\@tempdima 10
		\advance\dimen@-\@tempdima
		\ifdim\dimen@>\z@ \dimen@0pt\fi
		\rel@kern{0.6}\kern-\dimen@
		\if#31
		\overline{\rel@kern{-0.6}\kern\dimen@\macc@nucleus\rel@kern{0.4}\kern\dimen@}%
		\advance\dimen@0.4\dimexpr\macc@kerna
		\let\final@kern#2%
		\ifdim\dimen@<\z@ \let\final@kern1\fi
		\if\final@kern1 \kern-\dimen@\fi
		\else
		\overline{\rel@kern{-0.6}\kern\dimen@#1}%
		\fi
	}%
	\macc@depth\@ne
	\let\math@bgroup\@empty \let\math@egroup\macc@set@skewchar
zen@everymath{\mathgroup\macc@group\relax}%
	\macc@set@skewchar\relax
	\let\mathaccentV\macc@nested@a
	\if#31
	\macc@nested@a\relax111{#1}%
	\else
	\def\gobble@till@marker##1\endmarker{}%
	\futurelet\first@char\gobble@till@marker#1\endmarker
	\ifcat\noexpand\first@char A\else
	\def\first@char{}%
	\fi
	\macc@nested@a\relax111{\first@char}%
	\fi
	\endgroup
}
\makeatother

\makeatletter
\newcommand{\oset}[3][0.2ex]{%
	\mathrel{\mathop{#3}\limits^{
			\vbox to#1{\kern-2\ex@
				\hbox{$\scriptstyle#2$}\vss}}}}
\makeatother
\newcommand{\timesL}{\oset{L}{\times}}
\newcommand{\timesP}{\oset{P}{\times}}

\newcommand{\vlam}{{\check{\lambda}}}
\newcommand{\lam}{\lambda}
\newcommand{\vmu}{{\check{\mu}}}

\newcommand{\valpha}{{\check{\alpha}}}
\newcommand{\vbeta}{{\check{\beta}}}
\newcommand{\vlG}{\check{\lambda}_G}

\newcommand{\vlL}{\check{\lambda}_L}

\providecommand{\b}{\beta}
\newcommand{\<}{\langle}
\renewcommand{\>}{\rangle}

\providecommand{\sst}{\mathsf{sst}}
\newcommand{\st}{\mathsf{st}}
\newcommand{\gr}{\mathsf{gr}}
\newcommand{\coZ}{{Z}^\mathrm{c}}

\newcommand{\ad}{\mathsf{ad}}
\newcommand{\pt}{\mathsf{pt}}

\newcommand{\acts}{\curvearrowright}

\newcommand{\inv}{^{-1}}

\newcommand{\red}{^\mathrm{red}}

\newcommand{\Ane}{A_{n\!/\!e-1}}

\usetikzlibrary{decorations.markings} 
\tikzset{
  closed/.style = {decoration = {markings, mark = at position 0.5 with { \node[transform shape, xscale = .8, yscale=.4] {/}; } }, postaction = {decorate} },
  open/.style = {decoration = {markings, mark = at position 0.5 with { \node[transform shape, scale = .7] {$\circ$}; } }, postaction = {decorate} }
}
\newcommand{\PGL}{\operatorname{PGL}}
\newcommand{\GL}{\operatorname{GL}}
\newcommand{\SL}{\operatorname{SL}}

\newcommand{\fg}{\mathfrak{g}}
\newcommand{\fp}{\mathfrak{p}}

\newcommand{\fh}{\mathfrak{h}}
\newcommand{\fl}{\mathfrak{l}}
\newcommand{\fk}{\mathfrak{k}}
\renewcommand\sc{\mathsf{sc}} 

\newcommand{\Bun}{\operatorname{Bun}}

\DeclareMathOperator{\Aut}{Aut}

\DeclareMathOperator{\Hom}{Hom}
\DeclareMathOperator{\End}{End}

\DeclareMathOperator{\Pic}{Pic}
\DeclareMathOperator{\Jac}{{Jac}}

\DeclareMathOperator{\Sp}{Sp}

\DeclareMathOperator{\Spec}{Spec}

\DeclareMathOperator{\id}{\mathsf{id}}
\DeclareMathOperator{\Id}{\mathsf{Id}}

\newcommand{\ind}{\operatorname{ind}}
\newcommand\can{\mathsf{can}}

\theoremstyle{theorem}
\newtheorem{thm}{\bf{Theorem}}[section]
\newtheorem{lemma}[thm]{Lemma}                                              
\newtheorem{cor}[thm]{Corollary}
\newtheorem{prop}[thm]{Proposition}

\theoremstyle{remark}
\newtheorem{rem}{Remark}

\theoremstyle{definition}
\newtheorem{defi}[thm]{Definition}

\newtheorem{proof*}{Proof}
\newtheorem*{thm*}{Theorem}

\newcommand{\refeq}[1]{(\ref{#1})}

\title[G-bundles on elliptic curves]{Revisiting the moduli space of semistable $G$-bundles over elliptic curves}
\author{Drago\c s Fr\u a\c til\u a}
\address{IRMA,
7 rue Ren\'e Descartes,
67084 Strasbourg Cedex}
\email{fratila@math.unistra.fr}

\begin{document}

\xyoption {all} 

\begin{abstract} We show that the moduli space of semistable $G$-bundles on an elliptic curve for a reductive group $G$ is isomorphic to a power of the elliptic curve modulo a certain Weyl group which depend on the topological type of the bundle.
This generalises a result of Laszlo to arbitrary connected components and recovers the global description of the moduli space due to Friedman--Morgan--Witten and Schweigert. 
The proof is entirely in the realm of algebraic geometry and works in arbitrary characteristic. 
\end{abstract}
\maketitle 

\section{Introduction}

\subsection{} The study of principal $G$-bundles on elliptic curves began with the seminal paper of Atiyah~\cite{Ati} where he gave a complete and beautiful description of all the semistable vector bundles. He didn't discuss the moduli space but one could have easily guessed from his results the precise statement. Let us denote by $\cM_r^{d}$ the moduli space of semistable vector bundles of rank $r$ and degree $d$ on en elliptic curve $E/\bC$. In case $r$ and $d$ are coprime Atiyah essentially proved that the determinant map
\begin{equation}\label{E:Atiyah-det-iso}
	\det:\cM_r^d\to\cM_1^d
\end{equation}
is an isomorphism of algebraic varieties.

In general, if we put $m=\gcd(r,d)$, we have 
\begin{equation}\label{E:iso-Tu}
	\cM_r^d\simeq E^m/\fS_m
\end{equation}
where $E$ is the elliptic curve and $\fS_m$ is the symmetric group on $m$ letters (the isomorphism is not canonical however). The isomorphisms \refeq{E:Atiyah-det-iso} and \refeq{E:iso-Tu} hold in any characteristic and the proof, in characteristic 0, appeared in a paper by Tu~\cite[Theorem 1]{Tu_ell}.

For a reductive group $G$  and an elliptic curve $E$ over an algebraically closed field $\bk$ of arbitrary characteristic we denote by $\cM_G^d$ the moduli space of semistable $G$-bundles on $E$ of topological type $d\in\pi_1(G)$. 
The main result of this note can be summarized a bit imprecisely as
\begin{thm*} For any $d\in\pi_1(G)$ there is an isomorphism of moduli spaces
	\[\cM_G^d\simeq \cM_{C_d}^{d'}/W_d\]
	where $C_d$ is a certain algebraic torus, $d'\in\pi_1(C_d)$ and $W_d$ a certain Weyl group, all depending strongly on $d$.
\end{thm*}

\begin{rem}
	The result is surely well known to experts in characteristic 0 by passing through flat connections or twisted representations of fundamental groups. 
	However, to the author's knowledge, this is the first entirely algebraic proof that works also in positive characteristic.
\end{rem}

Laszlo~\cite{Lasz} proved the above theorem over $\bC$ in the case $d=0$, generalizing thus the isomorphism \refeq{E:iso-Tu}. More precisely, he proved that
\[\cM_G^0\simeq \cM_T^0/W\]
where we have denoted by $T$ a maximal torus of $G$ and $W$ is the Weyl group. 
His proof is through a Birkhoff-Grothendieck type result which says that every semistable $G$-bundle of degree zero over an elliptic curve is an extension of line bundles of degree zero.
Looijenga has proved~\cite{Looij_ell} that the RHS above is a weighted projective space where the weights can be read off the combinatorics of the root system of $G$.

Concerning the other components of the moduli space, motivated by 2d-conformal field theory, Schweigert has shown in \cite{Schw} that for any given topological type, say $d\in\pi_1(G)$, there is another reductive group, call it $G_d$, such that $\cM_G^d\simeq \cM_{G_d}^0$ as differentiable varieties. His statements are in the realm of differential geometry but one could possibly find a more algebro-geometric approach. 

Another take on this problem has been given by Friedman--Morgan--Witten in a series of papers \cite{FM-I, FMW1, FMW2}. They have two approaches: one is analytic through flat bundles which is very hands-on and adapted to concrete computations, however not very suitable to questions regarding families and moduli spaces. In their second approach, which uses deformation theory and is algebraic in nature, they provided a description of $\cM_G^d$ as a weighted projective space, thus recovering also Looijenga's theorem. However, their method is very different from Laszlo's and the relation to line bundles is not transparent.

\subsection{}Our goal in this note is to give a  description of $\cM_G^d$  in arbitrary characteristic in terms of line bundles by generalising Laszlo's approach.  Let us explain how to arrive at the statement of our theorem and then the difficulties and the ideas that arise in proving it.

The first difficulty is to find what should replace the torus. This has been dealt with in \cite[Theorem 3.2]{Fratila2016}.
It was shown that for a reductive group $G$ and a topological type $d\in\pi_1(G)$, there is a Levi subgroup $L_d$ and a $d'\in\pi_1(L_d)$ such that every \emph{polystable} $G$-bundle comes from a \emph{stable} $L_d$-bundle of degree $d'$. The role of the Weyl group $W$ will be taken by the relative Weyl group $W_d:=N_G(L_d)/L_d$.

This provides us with a well defined map of moduli spaces $\ind:\cM_{L_d}^{d'}\to\cM_G^d$ that is moreover \emph{finite} and $W_d$-invariant.
The second difficulty is to prove that the quotient map is an isomorphism. This would follow immediately by Zariski's main theorem provided we knew the map to be separable. It turns out that the question of separability (generic smoothness) is rather non-trivial in positive characteristic. 

The next step is relating $\cM_{L_d}^{d'}$ to line bundles. Inspired by Atiyah's theorem, the natural choice is to take the determinant map $\det:L_d \to L_d/[L_d,L_d]=:C_d$ and to show that it induces an isomorphism of varieties
\[ \det:\cM_{L_d}^{d'}\to  \cM_{C_d}^{\det(d')}. \]

Notice that $C_d$ is an algebraic torus so $\cM_{C_d}^{\det(d')}$ is isomorphic to a certain power of the Jacobian of $E$.

We have arrived at the following diagram
\begin{equation}\label{E:intro-diagr-modsp}
\begin{tikzcd}
& \cM_{L_d}^{d'}  \dlar[swap]{\ind} \drar{\det}&\\
\cM_G^d & & \cM_{C_d}^{\det(d')}
\end{tikzcd}
\end{equation} 
and our main theorem follows by proving two things: (i) $\det$ is an isomorphism; (ii) $\ind$ is generically étale with Galois group $W_d$.

The solution to both issues comes from the same tool: an \emph{extra symmetry} on the diagram (\ref{E:intro-diagr-modsp}), namely the abelian variety\footnote{just a product of several copies of  $\Jac(E)$ and maybe a finite abelian group.} $\cM_{Z(L_d)}^0$ acts on both $\cM_{L_d}^{d'}$ and $\cM_{C_d}^{\det(d')}$ making the map $\det$ equivariant. Moreover, the action is transitive and by computing the (reduced) stabilizers we conclude that $\det$ is an isomorphism. 
In addition, the action on $\cM_{L_d}^{d'}$ is used to prove that $\ind$ is generically étale by constructing a generic enough\footnote{for the differential of $\ind$} $L_d$-bundle. 

\subsection{}Some of the advantages of this approach over those in \cite{FM-I, FMW1} are that it also works in positive characteristic and the proofs in this paper are uniform with respect to the Dynkin type of the group $G$ and its isogeny class. 
To obtain precise information on the groups $L_d$ above, we do use however some results from \cite{Fratila2016}, namely Corollary 4.3 and Section 4.2, that are done by inspecting the combinatorics of each root system. 
Whereas in \cite{FMW1, Schw} the approach is set-theoretical and the structure of differential or complex variety needs to be constructed, here we're always dealing with the moduli stack/space as an algebro-geometric object and the maps between them are defined by functoriality, thus we never need to define or compare algebraic structures on a manifold.

Under some numerical conditions on $G$ and $d$ it was proved in \cite{FMW2}, independent of Looijenga's result \cite{Looij_ell},  that the moduli space $\cM_G^d$ is isomorphic to a certain weighted projective space.
A shortcoming of our approach is that it doesn't permit us to get this isomorphism without using Looijenga's result. 

We do not address in this paper the existence or the construction of universal bundles since they rarely exist on \emph{moduli spaces}. 
Indeed, the universal bundle on $\cM_{L_d}^{d'}$, if it exists, which is a rather subtle question, doesn't descend to $\cM_G^d$. See also Remarks~\ref{R:univ bdle descends to M_G}, \ref{R:univ bdle on M_L} where a few more details are provided. For a more thorough discussion of universal bundles on $\cM_G^d$ or opens of it we invite the reader to look at \cite{FMW1}. 

\subsection{} Below we introduce the necessary notation and we formulate precisely our main theorem. 

We'll be working over an algebraically closed field $\bk$ of arbitrary characteristic, $E$ is a smooth projective curve of genus one over $\bk$ and $G$ is a reductive group over $\bk$. 
We fix a Borus $T\subset B\subset G$ and we denote by $\bX_*(T)$ the group of cocharacters of $T$. 
For a reductive group $G$, the algebraic fundamental group is defined to be the quotient of the lattice of cocharacters by the lattice of coroots ( see \cite{Hump-algGr}[Humphreys, Section 31.1]). An algebraic group $P$ is an extension of a reductive group by a unipotent subgroup. 
Given that  unipotent groups are (at least topologically, in char. $0$) simply connected, one can define the algebraic fundamental group of a linear group by considering its reductive quotient.

In particular, in the situation that is of interest to us, namely for a parabolic subgroup $B\subset P\subset G$, the algebraic fundamental group is defined by 
\[ \pi_1(P):=\pi_1(M)=\bX_*(T)/\<\valpha \text{ coroot of } M\>_{\bZ}, \]
where $M=P/Rad_u(P)$ is the Levi quotient of $P$. 
We'll denote by $\vlam_P$ an element of $\pi_1(P)$. 

Remark that with this definition of fundamental group we have $\pi_1(\SL_n)=1$,  $\pi_1(\PGL_n) = \bZ/n\bZ$ and  $\pi_1(\GL_n) = \bZ$. Similarly, one can easily see that $\pi_1(\Sp_{2n}) = 1$ and $\pi_1(SO_{2n}) = \bZ/4\bZ$ or $(\bZ/2\bZ)^2$ according to $n$ being odd or even (see for example \cite[Planche I-VIII]{Bour456}).

The choice of $B$ gives us a notion of positive coroots and hence a partial order on the cocharacter lattice $\bX_*(T)$: we say $\vlam\le \vmu$ if $\vmu-\vlam$ is a positive linear combination of positive coroots. It extends naturally to rational coefficients $\bX_*(T)_\bQ$. 

The above partial order induces a partial order on the fundamental group $\pi_1(P)$ for any parabolic subgroup$B\subset P\subset G$ which extends naturally to $\pi_1(P)_\bQ$.

We denote by $\Bun_G^{\sst}$ and by $\cM_G$ the moduli stack, respectively moduli space, of semistable $G$-bundles over $E$. Their connected components are labeled by elements of $\pi_1(G)$, see \cite{Hoff} . We'll write $\Bun_G^{\vlam_G, \sst}$ and $\cM_G^{\vlam_G}$ for such a connected component.  Each such connected component is of finite type. 

In \cite{BalPara} it was proved, under some restrictions on the characteristic of the field, that $\cM_G^{\vlam_G}$ exists as a \emph{normal projective} variety. More precisely, the existence and normality of the moduli space was proved in arbitrary characteristic in \cite{GLSS_G-bdles} (see Section 1.1 Main Theorem). For projectivity, in \cite[Section 1.2] {GLSS_G-bdles} some assumptions on the characteristic of the field was needed. However, Heinloth showed in \cite{Heinl_sstable, Heinl_add_sstable} that the projectivity holds over arbitrary fields.

We have a canonical map $\Bun_G^{\vlam_G,\sst}\to \cM_G^{\vlam_G}$ which identifies two semistable $G$-bundles if their associated polystable $G$-bundles\footnote{For a semistable $G$-bundle $\cF_G$, the associated polystable $G$-bundle is the unique closed point of $\overline{\{\cF_G\}}\subset \Bun_G^{\vlam_G,\sst}$.} are isomorphic and kills all the automorphisms. 

Here are the main results of this paper formulated precisely: 

\begin{thm}
	\label{T:thmGalois covering}Let $\vlam_G\in\pi_1(G)$ be a fixed topological type. Then there exists a Levi subgroup $L=L_{\vlam_G}\subset G$ (unique up to conjugation) and $\vlam_L\in\pi_1(L)$ with the following properties:
	\begin{enumerate}
		\item (\cite{Fratila2016})  the inclusion $L\subset G$ induces a well defined map $\cM_L^{\vlam_L} \to \cM_G^{\vlam_G}$ and all the semistable $L$-bundles in $\cM_L^{\vlam_L}$ are \underline{stable}, in particular the $S$-equivalence relation reduces to isomorphism classes.
		\item $\cM_L^{\vlam_L}\to \cM_G^{\vlam_G}$ is a finite map, generically Galois, with Galois group the relative Weyl group $W_{L,G}=N_G(L)/L$.
		\item the following natural map is an isomorphism
		\[\cM_L^{\vlam_L}/W_{L,G}\simeq \cM_G^{\vlam_G}.\]
	\end{enumerate}
\end{thm}
\begin{rem}
	In characteristic $0$ the above theorem can be deduced rather easily from our previous result \cite[Theorem 3.2]{Fratila2016}. However, in the course of the proof we prove a technical result (see Lemma~\ref{L:exist_reg_bd}) that allows one to extend the results of \cite{Fratila2016} to arbitrary characteristic.
\end{rem}

\begin{thm}
	\label{T:det is iso}Let $L$ and $\vlam_L$ be as in the previous theorem. The map 
	\begin{equation}
		\det:\cM_L^{\vlam_L}\to \cM_{L/[L,L]}^{\det(\vlam_L)}
	\end{equation}
	is an isomorphism.
\end{thm}

\begin{cor}
	Let $\vlam_G\in\pi_1(G)$ and $L,\vlam_L$ as in Theorem~\ref{T:thmGalois covering}. Then we have
	\[ \cM_G^{\vlam_G}\simeq \cM_{L/[L,L]}^{\det(\vlam_L)}/W_{L,G}. \]
\end{cor}
\begin{rem}
	For a torus $Z$ we have $\cM_{Z}^0\simeq\Pic^0(E)\otimes_\bZ \bX_*(Z)$ and we see therefore that $\cM_G^{\vlam_G}$ can be described in terms of line bundles and a Weyl group.
\end{rem}
In particular, this theorem recovers Laszlo's result since for $\vlam_G=0$ the Levi $L_0$ is just the maximal torus. It also recovers the result of Tu because for $G=\GL_n$ and $\vlam_G\equiv d$ we have that $L=(\GL_{n/m})^{m}$ and $W_{L,G}=\fS_m$, where $m:=\gcd(d,n)$. 
It is not possible to compare directly our description of $\cM_G^{\vlam_G}$ with the one of Schweigert \cite{Schw} or Friedman--Morgan--Witten \cite{FM-I, FMW1} since there's no obvious algebraic relationship between $\cM_{L_{\vlam_G}}^{\vlam_L}$ and $\cM_{G_{\vlam_G}}^0$ (in loc.cit. the relation was made through representations of fundamental groups). 
However, one can check easily that the Weyl group of $G_{\vlam_G}$ is the same as our relative Weyl group $W_{L,G}$ and the maximal torus of $G_{\vlam_G}$ corresponds to the center $Z(L_{\vlam_G})$. 
In the case of $\GL_n$ the \emph{isomorphism} between $\Bun_G^{\vlam_G,\sst}(E)$ and $\Bun_{G_{\vlam_G}}^{0,\sst}$ (and their coarse moduli spaces) is provided by Fourier-Mukai transforms. 
It would be very nice to see if one can extend the Fourier-Mukai transforms to more general reductive groups. 
This subject will be discussed elsewhere.

\subsection*{Acknoledgements.} 

I would like to thank the Max Planck Institut für Mathematik in Bonn, where part of this work was done, for providing excellent working conditions.

\section{Preliminaries}
\subsection{Notation} For some notation, see the last paragraph of the introduction. Here are a few more that we'll be using. 
By a $G$-bundle we mean a $G$-torsor in the fppf topology over the scheme/stack in question. Over a curve this is the same as étale $G$-torsors for $G$ a smooth group. If $\cF_G$ is a $G$-bundle over $B$ and $F$ is a quasi-projective variety with a $G$ action (e.g. a representation) then we denote by $F_{\cF_G}=\cF_G\oset{G}{\times}F$ the associated fiber space over $B$ with fiber $F$. In particular, if $V$ is a representation of $G$, we have the associated vector bundle $V_{\cF_G}$.

We'll denote by $X$ a smooth projective curve over $\bk$. 
When we say curve, we always mean a smooth projective curve over $\bk$. 
Some results and definitions make sense for any genus so we'll state them like that.

For an algebraic group $H$ we denote by $\bB H=\pt/H$ the classifying stack of $H$-bundles. 
We denote by $\Bun_G(X)$ the moduli stack of $G$-bundles on $X$ and by $\cM_G(X)$ the corresponding moduli space (existence in arbitrary characteristic is proved in \cite{GLSS_G-bdles}. 
Similarly for the other groups $T,B,P$, etc. When we omit $X$ and write $\Bun_G$ or $\cM_G$ we mean $\Bun_G(E)$ or $\cM_G(E)$ where $E$ is an elliptic curve. 

The algebraic fundamental group of a reductive group is defined to be the quotient 

The connected components of $\Bun_G(X)$ are labeled by $\pi_1(G)$ (see \cite{Hoff}).

Let us begin by giving some definitions and citing some results that we'll be using throughout the paper.

\subsection{The slope map}
Before giving the definition of the slope map let us recall some basic things about cocharacters. 
If $L$ is a reductive group with maximal torus $T$ then the center of $L$ can be described as 
\[ Z(L) = \bigcap_{\alpha\text{ root of } L} \ker(\alpha) \subset T .\]
The natural map
$Z(L)\hookrightarrow T$ induces a map on cocharacters
\[ \bX_*(Z(L))\to \bX_*(T) \to \bX_*(T)/\<\valpha\mid \valpha \text{ coroot of }L\>=\pi_1(L)\]
which upon tensoring by $\bQ$ provides an isomorphism 
\[ \bX_*(Z(L))_\bQ\simeq \pi_1(L)_\bQ. \]

This follows from the following two simple facts: $\pi_1(L)_\bQ = \pi_1(L/[L,L])_\bQ$ and $Z(L)\to L/[L,L]$ is a finite surjective map (of diagonalizable groups). 

\begin{defi}
	[see~\cite{Schi}] For a parabolic subgroup $B\subset P\subset G$ with Levi subgroup $L$ we define the slope map $\phi_P:\pi_1(P)\to \bX_*(T)_\bQ$ as follows
	\[\pi_1(P)\to \pi_1(P)_\bQ\simeq \bX_*(Z(L))_\bQ\to \bX_*(T)_\bQ\]
	where we indicated by a subscript $\bQ$ the tensoring  $\otimes_\bZ\bQ$.
\end{defi}
Let us give some examples of fundamental groups and slopes for a few parabolic subgroups.

If $G=\GL_n$ and $\vlam_i, i=1,\dots,n$ are the coordinate cocharacters of the diagonal matrices then $\pi_1(G)\simeq\bZ \vlam_1$ and $\phi_G(d\vlam_1) = \frac dn(\vlam_1+\dots+\vlam_n)$. 

Continuing the previous example, let $P\le \GL_n$ be the parabolic with blocks of size $(k,n-k)$. Then $\pi_1(P) \simeq \bZ \vlam_k\oplus \bZ\vlam_{k+1}$ and $\phi_P(d\vlam_k+e\vlam_{k+1}) = \frac dk (\vlam_1+\dots+\vlam_k)+\frac{e}{n-k}(\vlam_{k+1}+\dots+\vlam_n)$.

For another example, let $G$ be a groupe of type $G_2$. This group is adjoint and simply connected at the same time. The coroot lattice is generated by $\valpha$ and $\vbeta$. 
The notation of the roots is such that  $\<\alpha,\vbeta\> = -1$ and $\<\beta,\valpha\> = -3$.

Let $P=P_\valpha$ be the parabolic corresponding to the coroot $\valpha$. Then $\pi_1(P) \simeq \bZ\vbeta$ and we can compute $\phi_P(\vbeta) = \frac12 \valpha+\vbeta$.
For the other maximal parabolic $Q=P_\vbeta$ we have $\pi_1(Q) = \bZ\valpha$ and $\phi_Q(\valpha) = \valpha+\frac32 \vbeta$.

The slope map has some very nice properties and we refer the interested reader to \cite{Schi} for a thorough treatment.

\subsection{Semistability} 
\begin{defi}
	\label{D:reduction-str grp} Let $H\subset K$ be a pair of algebraic groups and let $\cF_K\to Y$ be a $K$-bundle over $Y$. 
	A reduction of $\cF_K$ to $H$ is a couple $(\cF_H,\theta)$ of an $H$-bundle and an isomorphism $\theta:\cF_H\oset{H}{\times} K\simeq \cF_K$. 
	Two reductions $(\cF_H,\theta),(\cF'_H,\theta')$ are equivalent if there is an isomorphism of $H$-bundles $\cF_H\to \cF_H'$ such that its extension to $K$ composed with $\theta'$ is equal to $\theta$.
\end{defi}	
\begin{rem}
	\label{R:reduction-section}To give a reduction of a $K$-bundle $\cF_K$ to $H$ is the same as to give a section of $\cF_K/H\to Y$. 
	Two such sections give equivalent reductions if and only if there exists an automorphism $\sigma\in\Aut(\cF_K)$ translating one into the other.
\end{rem}
\begin{rem}
	For example, if $K=\GL_n$ and $H$ is the subgroup of upper-triangular matrices, then to give a reduction to $H$ of a rank $n$ vector bundle  (i.e. a $\GL_n$-bundle) is the same as to give a filtration of it with sub-quotients being line bundles.
\end{rem}
The following definition of semistability for $G$-bundles is from \cite{Schi} where it is also proved the equivalence with the Ramanathan's semistability.

\begin{defi}
	A $G$-bundle $\cF_G$ of degree $\vlam_G$ over a smooth projective curve $X$ is (semi)stable if for any proper parabolic subgroup $P\subset G$ and for any reduction $\cF_P$ of $\cF_G$ to $P$ of degree $\vlam_P$ we have 
	\[\phi_P(\vlam_P)\underset{(\le)}{<} \phi_G(\vlam_G).\]
\end{defi}

\begin{prop}\label{P:slope assoc vbdl}\cite[Proposition 3.2 (b)]{Schi} 
	 If $V$ is a highest weight representation of $G$ of highest weight $\lambda$ and $\cF_G$ is a $G$-bundle of degree $\vlam_G$ over a curve $X$ then the slope (i.e degree divided by rank) of the associated vector bundle $V_{\cF_G}$ is 
	\[  \mu(V_{\cF_G})=\<\phi_G(\vlam_G),\lambda\>. \]
\end{prop}

\subsection{Frobenius semistability} In case $\bk$ is of characteristic $p$, there is a stronger notion of stability, called Frobenius semistability and it behaves better with respect to associated vector bundles.
 
Denote by $F_X:X\to X$ the absolute Frobenius: it is the identity at the level of topological spaces and raising to the power $p$ at the level of functions. 
\begin{defi}
	A $G$-bundle $\cF_G$ is Frobenius semistable if $(F^n)^*(\cF_G)$ is semistable for all $n\ge 0$.
\end{defi}

In characteristic zero we have the following remarkable property: the tensor product of two semistable vector bundles of the same slope is again semistable. 
The correct analogue in characteristic $p$ is the following:

\begin{lemma}
	\label{L:XS-stg-sem-induced}\cite[Corollary 1.1]{Sun_sstab} 
	Let $\cF_G$ be a \emph{Frobenius semistable} $G$-bundle over a smooth projective curve $X$ and let $f:G\to G'$ be a morphism of reductive groups such that $f(Z(G))\subset Z(G')$. Then the induced $G'$-bundle is also Frobenius  semistable. In particular, if $V$ is a representation of $G$ such that the center of $G$ acts by a character, the induced vector bundle $V_{\cF_G}$ is semistable.
\end{lemma}
This result is relevant to us because of the following theorem:
\begin{thm}\label{T:Sun-thm sstb elliptic}\cite[Theorem 2.1]{Sun_sstab}
	For curves of genus one semistability and Frobenius semistability are equivalent notions.	
\end{thm}
These two put together give

\begin{cor}\label{C:induced is semistable}
	Let $\cF_G$ be a semistable $G$-bundle over an elliptic curve and let $V$ be a representation of $G$ such that the center of $G$ acts by a character. Then the vector bundle $V_{\cF_G}$ is semistable.
\end{cor}

The above Corollary is crucially used in the proof of \cref{L:exist_reg_bd}.

\subsection{Jordan-H\"older series}
In the case of vector bundles it makes sense to talk about the category of semistable vector bundles of fixed slope. 
This is a finite length category so we can also talk about Jordan-Hölder series. To give a filtration of a vector bundle is the same as to give a reduction of the corresponding $\GL_n$-bundle to a certain parabolic subgroup. In general, the Jordan-Hölder series has no reason to have the same slopes of the graded parts when the vector bundle varies. However, this is a particularity of elliptic curves. Namely, it can be extracted from Atiyah's paper \cite{Ati} that for semistable vector bundles of rank $n$ and degree $d$ there is a (unique up to conjugation) parabolic subgroup such that all the semistable vector bundles of rank $n$ and degree $d$ admit a reduction to it and moreover the graded parts are stable vector bundles of equal slope. For example, for slope $0$, all semistable vector bundles are extensions of degree zero line bundles.

The following is an analogue for any reductive group $G$ and any degree $\vlam_G$.

\begin{thm}
\label{T:JH for Gbdles}\cite[Lemma 2.12, Theorem 3.2, Corollary 4.2]{Fratila2016} 
	Let $\vlam_G\in \pi_1(G)$ and consider $\Bun_G^{\vlam_G,\sst}$ the stack of semistable $G$-bundles of degree $\vlam_G$ on an elliptic curve $E$. Then there exists a unique (up to conjugation) parabolic subgroup $P$ and a unique $\vlam_P\in\pi_1(P)$ such that 
	\begin{enumerate}
		\item $\phi_G(\vlam_G)=\phi_P(\vlam_P)$,
		\item every semistable $G$-bundle of degree $\vlam_G$ has a reduction to $P$ of degree $\vlam_P$,
		\item the map
		\[\Bun_P^{\vlam_P,\sst}\to\Bun_G^{\vlam_G,\sst}\]
		is proper, generically Galois with Galois group $W_{L,G}=N_G(L)/L$ where $L$ is the Levi subgroup of $G$.
		\item for any $\cF_P\in\Bun_P^{\vlam_P,\sst}$ the induced $L$-bundle is stable.
		\item (\cite[Corollary 4.3]{Fratila2016}) For a reductive group $L$ and $\vlam_L\in\pi_1(L)$ there exist stable $L$-bundles of degree $\vlam_L$ if and only if $L^\ad=\prod_i \PGL_{n_i}$ and $\vlam_L^\ad\equiv (d_i)_i$ with $\gcd(d_i,n_i)=1,~\forall i$.			
	\end{enumerate}
\end{thm}

\begin{rem}
	In \cite{Fratila2016} there is a table with all the possible subgroups $L$ that appear in the above theorem. For the convenience of the reader we provide a copy of the table in the Appendix.
\end{rem}
\begin{rem}
	The proof from \cite{Fratila2016} is in characteristic zero, however the only moment that we used it was to apply "generic smoothness" (see \cite[Lemma 3.9]{Fratila2016}) and deduce the existence of certain regular bundles (see Definition~\ref{D:reg bundles}) which we prove here in arbitrary characteristic (see Lemma~\ref{L:exist_reg_bd}).
Therefore the results of \cite{Fratila2016} hold in positive characteristic as well.
\end{rem}

\subsection{Vector bundles over elliptic curves}

\begin{thm}
\label{T:At-coprime-unique-det}\cite[Corollary to Theorem 7]{Ati} Let $n\ge 1$ and $d\in\bZ$ be coprime. 
	\begin{enumerate}
		\item Any stable rank $n$ degree $d$ vector bundle over $E$ is uniquely determined by its determinant bundle.
		\item If $\cV$ is a vector bundle as above and $\cL\in\Pic^0(E)$ then $\cV\otimes\cL\simeq\cV$ if and only if $\cL\in\Pic^0(E)[n]$, the $n$-torsion subgroup.
	\end{enumerate}
\end{thm}

\begin{thm}
	\label{T:BH-Bun_Z-torsor}\cite[Lemma 2.2.1 and Example 5.1.4]{BisHoff-lbdles} Let $X$ be a smooth projective curve and let \[1\to Z\to G\to H\to 1\] be a central extension. Fix $\vlam_G\in\pi_1(G)$ and denote by $\vlam_H$ the image of $\vlam_G$ in $\pi_1(H)$. Then the map
	\[\Bun_G^{\vlam_G}(X)\to\Bun_{H}^{\vlam_H}(X)\]
	is a $\Bun_Z^0(X)$-torsor. 
\end{thm}
\begin{rem}
	The same holds for semistable bundles also since being semistable for $G$ or $H$ is the same thing (the flag varieties are the same).
\end{rem}

\begin{cor}
	\label{C:M_PGLn=pt}Let $n\ge 1$ and $d\in\bZ$ be coprime. Then over an elliptic curve $E$ we have 
	\[\Bun_{\PGL_n}^{d,\st}\simeq \bB\Pic^0(E)[n]\red,\]
	where we have denoted by $\Pic^0(E)[n]$ the kernel (subgroup scheme) of the multiplication by $n:\Pic^0(E)\to \Pic^0(E)$.
	In particular, we deduce that $\cM_{\PGL_n}^d=\pt$.
\end{cor}

\begin{proof}
	Using Theorem \ref{T:BH-Bun_Z-torsor} we have that $\Bun_{\GL_n}^{d,\st}\to\Bun_{\PGL_n}^{d,\st}$ is a $\Bun_{\bG_m}^0$-torsor. 
	By Theorem~\ref{T:At-coprime-unique-det} we deduce that $\Bun_{\PGL_n}^{d,\st}$ has only one isomorphism class of objects and the automorphism group is the kernel of the action of $\Pic^0(E)$ on $\Bun_{\GL_n}^{d,\st}$. However, this kernel must be a smooth group scheme because $\Bun_{\PGL_n}$ is a smooth stack, so by Theorem~\ref{T:At-coprime-unique-det}  it must be  $\Pic^0(E)[n]\red$ .
\end{proof}

\begin{rem}
	I don't know a direct way of showing that the scheme theoretic stabilizer of the action of $\Pic^0(E)$ on $\Bun_{\PGL_n}^{d,\sst}$ is precisely $\Pic^0(E)[n]\red$.
\end{rem}

\section{Proof of \cref{T:thmGalois covering}}

\subsection{The action of the center}
\begin{cor}
	\label{C:M_Z acts trans on M_L} Let $E$ be an elliptic curve, let $L$ be a reductive group and let $\vlam_L\in\pi_1(L)$ such that there exist stable $L$-bundles of degree $\vlam_L$ on $E$ (see Theorem~\ref{T:JH for Gbdles} (5) ). Then the action of $\cM_{Z(L)}^0$ on $\cM_L^{\vlam_L}$ is transitive.
\end{cor}
\begin{proof}
	From Theorem~\ref{T:JH for Gbdles} (5) we have  $L^{\ad}\simeq\prod_i \PGL_{n_i}$ and $\vlam_L^{\ad}=(d_i)_i$ such that $\gcd(d_i,n_i)=1$. So we can apply Corollary~\ref{C:M_PGLn=pt} to conclude that $\cM_{L^\ad}^{\vlam_L^\ad}=\pt$. 
	
	Since $\Bun_L^{\vlam_L,\st}\to \Bun_{L^\ad}^{\vlam_L^\ad,\st}$ is a $\Bun_{Z(L)}^0$-torsor (see Theorem~\ref{T:BH-Bun_Z-torsor}) we deduce that $\Bun_{Z(L)}^0$ acts on $\Bun_L^{\vlam_L,\st}$ transitively on objects. This property is clearly preserved when we pass to moduli spaces.
\end{proof}
\begin{cor}
	\label{C:Bun_L same autom group} 
	Under the  hypotheses of the previous corollary, all $L$-bundles $\cF_L$ in $\Bun_L^{\vlam_L,\st}$ have the same automorphism group.
\end{cor}	

\begin{proof}
	We put $Z:=Z(L)$. For $\cF_Z\in\Bun_Z^0$ and $\cF_L\in\Bun_L$ there is a canonical isomorphism $\Aut(\cF_L)\to\Aut(\cF_L\otimes\cF_Z)$ sending $\theta$ to $\theta\otimes \id$. From Corollary~\ref{C:M_Z acts trans on M_L} the action $\Bun_Z^0\acts \Bun_L^{\vlam_L,\st}$ is transitive on objects so we conclude.
\end{proof}

\begin{rem}
	The above Corollary is never used in the sequel but it allows us to see that $\Bun_L^{\vlam_L,\st}\to \cM_L^{\vlam_L}$ is a gerbe.
\end{rem}

\subsection{Regular bundles}
This subsection is dedicated to proving the following Lemma which was one of the key obstacles:
\begin{lemma}\label{L:gen etale} 
	Let $\vlam_G\in\pi_1(G)$ and $L,P,\vlam_L\in\pi_1(L)$ be as in Theorem~\ref{T:JH for Gbdles}. 
	Then the map $\ind:\cM_L^{\vlam_L}\to \cM_G^{\vlam_G}$ is generically étale.
\end{lemma}

Let us introduce the notion of regular $L$-bundles.\footnote{There exists another notion of regular stable bundles: those whose automorphism group is exactly the center of the group (see \cite{FM-I, FMW1}). However we'll not use this notion in this paper.}

\begin{defi}\label{D:reg bundles}
	\begin{enumerate}
		\item Let $H$ be an algebraic group and $V$ a representation of $H$. Consider $\vlam_H\in\pi_1(H)$ such that its image in $\pi_1(\GL(V))$ is $0$. An $H$-bundle $\cF_H$ of degree $\vlam_H$ is called $V$-\emph{regular} if $H^0(X,V_{\cF_H})=0$,
		\item Let $P\subset G$ be a parabolic subgroup with Levi subgroup $L$. A $P$-bundle over a curve $X$ is called \emph{regular} if it is $\fg/\fp$-regular. An $L$-bundle is \emph{regular} if it is $\fg/\fl$-regular.
	\end{enumerate}
\end{defi}

\begin{rem}
This condition on $\cF_P$ is in order for the differential of $p:\Bun_P^{\vlam_P}\to\Bun_G^{\vlam_G}$ to be injective at $\cF_P$.
However, Serre duality over elliptic curves implies also the surjectivity, i.e. smoothness of $p$ at $\cF_P$.
\end{rem}

The core of the proof of Lemma~\ref{L:gen etale} is to show that there exist regular bundles:
\begin{lemma}
	\label{L:exist_reg_bd}
	Let $X$ be a curve and $\vlam_G,P,L,\vlam_P$ as in Theorem~\ref{T:JH for Gbdles}. Then the substack of regular $L$-bundles in $\Bun_L^{\vlam_L,\st}$ is open and dense.
\end{lemma}
\begin{proof}
	The strategy is the following: we start with an arbitrary Frobenius semistable $L$-bundle (see Theorem~\ref{T:Sun-thm sstb elliptic} for existence) and we tensor it with a sufficiently generic $Z:=Z(L)$-bundle of degree zero to produce a regular $L$-bundle.

	The openness follows from the semi-continuity of $\dim(H^0(X,(\fg/\fl)_{\cF_L}))$ so all we need to prove is the non-emptiness of the regular locus.
	
	More precisely, let $\cF_L$ be a Frobenius stable $L$-bundle of degree $\vlam_L$ and $V$ be a highest weight representation of $L$ such that $V_{\cF_L}$ is of degree zero and such that the center $Z=Z(L)$ acts on $V$ by a nontrivial character $\chi$. Corollary~\ref{C:induced is semistable} guarantees that $V_{\cF_L}$ is semistable of degree zero and hence the set of isomorphism classes of line subbundles of degree zero of $V_{\cF_L}$ is finite. 
	
	Now let us consider a $Z$-bundle $\cF_Z$  of degree zero. Using the group morphism $Z\times L\to L$ we can produce a new $L$-bundle that we denote $\cF_L\otimes \cF_Z$ which is still Frobenius semistable of degree $\vlam_G$. 
	The center $Z$ acts on $V$ by $\chi$ so we have that $V_{\cF_L\otimes \cF_Z}=V_{\cF_L}\otimes \chi_{\cF_Z}$, hence the set of line subbundles of degree zero of $V_{\cF_L\otimes \cF_Z}$ is the one for $V_{\cF_L}$ tensored by $\chi_{\cF_Z}$. 
	Since $\chi$ is non-trivial, we obtain that for almost all $Z$-bundles the trivial line bundle $\cO$ is not a line subbundle of $V_{\cF_L\otimes \cF_Z}$, in other words  $H^0(X,V_{\cF_L\otimes\cF_Z})=0$. So we've produced an open dense substack of $L$-bundles $\cF_L$ of degree $\vlam_P$ that are $V$-regular.
	
	Let us apply the previous paragraph to the representation $L\acts \fg/\fl$. It is not a highest weight representation but it admits a filtration with subquotients of highest weight. 
	Let $W$ be such a subquotient. It is of highest weight, say $\alpha$, that belong to the roots of $\fg$.
	By Proposition~\ref{P:slope assoc vbdl} we have $\deg(W_{\cF_L}) = \<\phi_L(\vlam_P),\alpha\>$. By hypothesis (see Theorem~\ref{T:JH for Gbdles} (1)) we get $\deg(W_{\cF_L}) = \<\phi_G(\vlam_G),\alpha\>$ which is zero because by the definition of the slope map $\phi_G(\vlam_G)$ is a cocharacter of the center of $G$, hence it vanishes on the roots of $G$.	
	
	As the weights of $\fg/\fl$ are among the roots of $\fg$, we see that if $W$ is such a subquotient, then $W_{\cF_L}$ is semistable of degree zero (see Lemma~\ref{L:XS-stg-sem-induced} and Proposition~\ref{P:slope assoc vbdl}). Also the central characters are not trivial because the centraliser of $Z(L)$ in $G$ is precisely $L$. Therefore, by the previous paragraph applied to each such subquotient $W$, the substack of $W$-regular $L$-bundles is open and dense and so is their intersection (finite number) which is nothing else than the substack of regular $L$-bundles.
\end{proof}

\begin{proof}
	(of Lemma~\ref{L:gen etale}) Proving generic étaleness is equivalent to proving the map is étale at some point, say $\cF_L$. By looking at the differential of the map we have to show the bijectivity of
	\[H^1(E,\fl_{\cF_L})\to H^1(E,\fg_{\cF_L}).\]
	This is implied by the vanishing of $H^i(E,(\fg/\fl)_{\cF_L}),~i=0,1$. 
	
	Let $\cF_L$ be a regular $L$-bundle (see Lemma~\ref{L:exist_reg_bd} ). Then by definition we have $H^0(E,(\fg/\fl)_{\cF_L})=0$. By Riemann-Roch we get that $H^1(E,(\fg/\fl)_{\cF_L})=\deg((\fg/\fl)_{\cF_L})=0$ where for the last equality we used genus one and Proposition~\ref{P:slope assoc vbdl}.
\end{proof}

\begin{lemma}\label{L:W-invar-W orbits}
	Let $\vlam_G\in\pi_1(G)$ and $L,P,\vlam_L\in\pi_1(L)$ as in Theorem~\ref{T:JH for Gbdles} and put $W:=W_{L,G}$ the relative Weyl group of $L\subset G$. Then the map $\pi:\cM_L^{\vlam_L}\to \cM_G^{\vlam_G}$ is $W$-invariant and the fibers are $W$-orbits. In particular it is a finite map.
\end{lemma}
\begin{proof}
	Both moduli spaces are projective varieties so finiteness follows from quasi-finiteness which in turn follows from the fact that the fibers are $W$-orbits.
	
	Remark that the map $\pi$ is clearly $W$-invariant. Indeed, this is a general fact: an $H$-bundle doesn't change its isomorphism class when acted upon by an inner automorphism of $H$. In our case, the action of an element $w\in W=N_G(L)/L$ on $L$ becomes an inner automorphism of $G$, so the isomorphism class of the induced $G$-bundle is not affected.
	
	Let us prove now that the fibers are $W$-orbits. Let $\cF_L,\cF_L'\in\cM_L^{\vlam_L}$ be two $L$-bundles in the fiber of $\pi$, namely $\cF_L\timesL G\simeq \cF'_L\timesL G$. Let us call $\cF_P$ and $\cF'_P$ the induced $P$-bundles. Notice that we can recover the $L$-bundles as $\cF_P/U=\cF_L$ and $\cF'_P/U=\cF'_L$ where $U$ is the unipotent radical of $P$.

	The $P$-bundle $\cF'_P$ is a reduction of the $G$-bundle $\cF_P\timesP G$ to $P$ so by Remark~\ref{R:reduction-section} we can think of $\cF'_P$ as given by a section $s:E\to \cF_P\timesP G/P$.
	
	By the Bruhat decomposition we have $G = \bigsqcup_{ w\in W_L\backslash W_G/W_L} PwP$ where $W_L = N_L(T)/T$ and $W_G=N_G(T)/T$ are the corresponding Weyl groups of $L$ respectively $G$.

	Let us recall the notion of relative position\footnote{it is a generalization of the notation of relative position of two flags of a vector space} for two reductions to $P$: we say that the bundles $\cF_P$ and $\cF'_P$ are (generically) in relative position $w\in W_L\backslash W_G/W_L$ if the section $s:X\to \cF_P\timesP G/P$ lands (generically) in $\cF_P\timesP P  w P/P$.

	Let us denote by $w\in N_G(T)/T$ a lift of the generic relative position of $\cF_P$ and $\cF'_P$ which we recall are both of degree $\vlam_P$ (see Lemma~\ref{T:JH for Gbdles}). 
	Semistability and the assumption on the degree allow us to use Lemma 3.5 from \cite{Fratila2016} and obtain that the two $P$-bundles are in relative  position $w$  everywhere (not merely generically) and then Lemma 3.7 from loc.cit gives us moreover that $w\in N_G(L)/L=W_{L,G}$. 
	
	To summarize, we have $s:E\to \cF_P\timesP PwP/P$ and the bundle $\cF'_P$ is given by pullback
	\[ \xymatrix{
	\cF'_P\ar@{^(->}[rr]^{i}\ar[d] & &\cF_P\timesP PwP\ar[d]\\
	E\ar[rr]^{s} & &\cF_P\timesP PwP/P
} \]
	
	The quotient map $PwP\to LwL$ is $P\times P$ equivariant so using it in the above diagram and composing with $i$ we get a $P$-equivariant map of bundles $\cF'_P\to \cF_P\timesP LwL$ . 
	It clearly factors through $\cF'_P/U=\cF'_L$ and since the action of $P\times P$ on $LwL$ factors through $L\times L$ we obtain an $L$-equivariant map of $L$-bundles $\cF'_L\to \cF_L\timesL LwL$. 
	It remains to notice that $\cF_L\timesL LwL$ is none other than $w^*(\cF_L)$ and the proof is complete.
\end{proof}

\subsection{Proof of Theorem~\ref{T:thmGalois covering}}
\begin{proof}
	We finish the proof of Theorem~\ref{T:thmGalois covering}.
	The point (1) is contained in Theorem~\ref{T:JH for Gbdles} (4). 
	
	To prove (2) we combine Theorem~\ref{T:JH for Gbdles} (3), Lemma~\ref{L:gen etale} and Lemma~\ref{L:W-invar-W orbits}.
	
	To prove (3), from Lemma~\ref{L:W-invar-W orbits} we have that the natural map $\cM_L^{\vlam_L}\to\cM_G^{\vlam_G}$ factorises through $\cM_L^{\vlam_L}/W_{L,G}$ and moreover the morphism $\cM_L^{\vlam_L}/W_{L,G}\to \cM_G^{\vlam_G}$ is bijective and separable, see Lemma~\ref{L:gen etale}. Since the target is a normal variety (see \cite{GLSS_G-bdles}), we can apply Zariski's main theorem to conclude that it is an isomorphism.
\end{proof}
\begin{rem}\label{R:univ bdle descends to M_G} 
	One might think that if $\cM_L^{\vlam_L}$ has a universal bundle then it descends to $\cM_G^{\vlam_G}$. However, unless $L=G$, this is not the case and one reason is that the dimension of the automorphism group of a $G$-bundle induced from $L$ varies (the jumps arise at non regular $L$-bundles).
\end{rem}

\section{Proof of Theorem~\ref{T:det is iso}} 
Unless otherwise stated, in this section $L$ is a reductive group and $\vlam_L$ is an element of $\pi_1(L)$ such that there exist stable $L$-bundles of degree $\vlL$ where $E$ is an elliptic curve. 
In this section we'll prove \cref{T:det is iso} which asserts that $\det$ is an isomorphism of varieties
\[ \det:\cM_L^{\vlam_L} \to \cM_{L/[L,L]}^{\det\vlam_L} .\]

The idea of the proof is rather simple: we show that the map is bijective on $\bk$-points by exploiting the action of $\cM_{Z(L)}^0$ on 
both varieties; using the differential criterion we show that it's also étale. We have thus a finite, étale map of degree 1, hence an isomorphism.
\subsection{Preliminaries} Recall from Theorem~\ref{T:JH for Gbdles} (5) that the assumption on $L$ forces $L^\ad\simeq\prod_i \PGL_{n_i}$ for some $n_i$. We denote by $\coZ=L/[L,L]$ the co-center of $L$ and by $Z=Z(L)$ the center of $L$.

Let us recall that the natural map $\det:L\to \coZ$ is called the determinant. The homomorphisms $Z\times L\to L$ and $Z\times \coZ\to\coZ$ naturally give actions of $\Bun_{Z}^0$ on $\Bun_L^{\vlam_L,\st}$ and on $\Bun_{\coZ}^{\det(\vlam_L)}$.

A diagonalizable group is a linear algebraic group that is isomorphic to a product of several $\bG_m$ and $\mu_n$ for various $n\ge 2$. The category of diagonalizable groups is anti-equivalent to the category of finitely generated abelian groups, where the functors are given by $D\mapsto \Hom_\gr(D,\bG_m)$ and $\Lambda\mapsto \Spec(k[\Lambda])$.

For a diagonalizable group $D$, we write $\Bun_D^0(X)$ for the moduli stack of $D$-bundles $\cF_D$ on $X$ of degree zero, that is such that for any character $\chi:D\to\bG_m$ the associated line bundle $\chi_{\cF_D}$ is of degree zero. If $D$ is not a torus, then this stack might not be connected. For example if the characteristic of $\bk$ doesn't divide $n$ then for $D=\mu_n$ we have $\Bun_{\mu_n}^0(X)=\Pic^0(X)[n]\times\bB\mu_n$, where $\Pic^0(X)[n]$ is the $n$-torsion in $\Pic^0(X)$. 

We denote by $\cM_D^0(X)$ the moduli \emph{space} of $D$-bundles of degree zero on $X$ in the same sense as above. 
It is a group scheme whose (reduced) connected component of the identity is an abelian variety. 
For example, if $D$ is a torus, then $\cM_D^0\simeq\Pic^0(X)^{\dim(D)}$. 
If $D$ has some finite component then $\cM_D^0$ is a product of an abelian variety and a finite group scheme which is a finite subgroup of an abelian variety. 
For example, for $D=\mu_n$ we have  $\cM_D^0(X)=\Pic^0(X)[n]$. 
Remark that in positive characteristic $\cM_D^0(X)$ might not be reduced.

Here is a basic general lemma:
\begin{lemma}\label{L:iso descends to reductions} 
	Let $H\subset L$ be reductive groups such that $[H,H]=[L,L]$. Let $\cF_H,\cF'_H$ be two $H$-bundles on a proper scheme $Y$. Then if the induced $L$-bundles are isomorphic, the $H$-bundles are also.
\end{lemma}
\begin{proof}
	We will see $\cF_H'$ as a reduction to $H$ of the $L$-bundle $\cF_H\oset{H}{\times} L$ (see Definition~\ref{D:reduction-str grp} and the remark following). So we have a section $s:Y\to \cF_H\oset{H}{\times} L/H$. We need to show that by an automorphism of the $L$-bundle $\cF_H\oset{H}{\times}L$ we can translate it into the trivial section $s_0:Y=\cF_H\oset{H}{\times}H/H\hookrightarrow \cF_H\oset{H}{\times}L$.
	
	From the assumptions we have $H Z(L)=L$ hence $H$ acts trivially on $L/H$. Therefore the section $s$ can be seen as a section $s:Y\to Y\times L/H$, i.e. as a map $Y\to L/H$. As $H, L$ are reductive the quotient $L/H$ is an affine variety so $s$ must be constant, say equal to $\overline z$,  because $Y$ is proper.
	
	The assumptions imply the surjectivity $Z(L)\twoheadrightarrow L/H$ so we can take $z\in Z(L)$ a lift of $\overline z$. 
	The element $z$ being in the center of $L$ gives an automorphism, call it $\theta_z$, of $\cF_H\oset{H}{\times}L$ such that $\theta_z^{-1}(s)=s_0$. 
	In other words, the section $s$ gives an $H$-bundle isomorphic to $\cF_H$. 
\end{proof}

\begin{rem}
	The above Lemma is false if $Y$ is not proper (think of modules over a Dedeking ring) and it is also false if $L/H$ is not affine (two filtrations of the same vector bundle need not be isomorphic).
\end{rem}

\subsection{Diagonalizable groups} We collect here some technical lemmas on diagonalizable groups and bundles over a smooth projective curve $X$. 
We advise the reader to skip this section and come back to it when it is referred to.

Some notation: if $C$ is a normal subgroup of $G_1$ and $G_2$ then $C\hookrightarrow G_1\times G_2$ given by $c\mapsto (c,c\inv)$ is a normal subgroup. 
We denote by $G_1\oset{C}{\times}G_2$ the quotient group $(G_1\times G_2)/C$.

\begin{lemma}\label{L:embed [L,L] s.c.}
	Let $L$ be a reductive group such that $L^\ad\simeq\prod_i\PGL_{n_i}$. Assume that $[L,L]$ is simply connected. Then there exists a torus $T'$ such that $L\hookrightarrow \prod_i \GL_{n_i}\times T'$.
\end{lemma}
\begin{proof} The proof is essentially linear algebra.
	
We have $L=\prod_i \SL_{n_i}\oset{C}{\times} Z(L)$, where $C=\prod_i \mu_{n_i}$. We put $\can:C\hookrightarrow  \prod_i \bG_m$ the canonical inclusion.

There is a torus $T'$ and a map $\phi:Z(L)\hookrightarrow \prod_i \bG_m\times T'$ such that the following diagram commutes:
\[
\begin{tikzcd}
C \rar[hook] \dar[hook, swap]{\can} & Z(L) \dlar[hook]\\
\prod_i \bG_m \times T' 
\end{tikzcd}
\]

Indeed, using the equivalence of diagonalizable groups with finitely generated abelian groups, we need to show that it exists $\phi:\prod_i \bZ \times \bZ^r \twoheadrightarrow M$ such that the following diagram commutes 
\[\xymatrix{
\prod_i \bZ/n_i & M \ar@{->>}[l]^u\\
\prod_i \bZ \times \bZ^r \ar@{->>}[u]^{\can} \ar@{->>}[ur]_\phi
}\]
where $M$ is the abelian group of characters of $Z(L)$.

This can be done easily as follows: first take $r=0$ and use that $\prod_i \bZ$ is free and $u$ is surjective. Then, for a convenient $r\ge 0$ add $\bZ^r$ mapping surjectively onto $\ker(u)$.
\end{proof}

\begin{lemma}
	\label{L:proj to L when [L,L] non s.c.} Let $L$ be an arbitrary reductive group. Then there exists a central extension
	\[1\to T'\to \hat{L}\to L\to 1\]
	with $[\hat L,\hat L]$ simply connected and $T'$ a torus. In particular, since $T'$ is connected, we have $\pi_1(\hat L)\twoheadrightarrow \pi_1(L)$.
\end{lemma}
\begin{proof}
	We write $L=[L,L]\oset{C}{\times}Z=[L,L]^\sc\oset{\tilde{C}}{\times}Z$ where $\tilde{C}=Z([L,L]^\sc)$ and $[L,L]^\sc$ is the simply connected cover of $[L,L]$. 
	Let us choose a torus $T'$ and an inclusion $\tilde{C}\hookrightarrow T'$. We define the following group
	\[\hat L:=[L,L]^\sc\oset{\tilde C}{\times}(Z\times T')\]
	where $\tilde{C}\to Z\times T'$ is the diagonal homomorphism (injective!). Clearly $[\hat L,\hat L]=[L,L]^\sc$.
	The natural homomorphism $\hat L\to L$, forgetting the factor $T'$, is surjective and its kernel is exactly $T'$.
\end{proof}

\begin{lemma}\label{L:inj-diaggrps-inject}
	Let $Z\hookrightarrow Z'$ be an injective morphism of diagonalizable groups. Then 
	\[ \Bun_Z^0(X)\to \Bun_{Z'}^0(X) \text{ is injective on objects.} \]
\end{lemma}
\begin{proof}
	Let $\cF,\cF'$ be two $Z$-bundles such that there exists 
	\[ \theta:\cF\oset{Z}{\times}Z'\simeq \cF'\oset{Z}{\times}Z' \text{ isomorphism of } Z'\text{-bundles}.  \]
	This is equivalent to having 
	\[ \theta:\cF\to \cF'\oset{Z}{\times}Z' \text{ a } Z\text{-equivariant bundle map.} \]	
	Taking the quotient by $Z$ we obtain 
	\[ \overline{\theta}:X\to \cF'\oset{Z}{\times}Z'/Z=X\times Z'/Z \text{ a morphism over } X. \]
	Since $Z'/Z$ is affine (diagonalizable groups) and $X$ is proper, the map $\overline{\theta}$ must be constant, say equal to $z_0Z$. 
	Due to the commutativity of the groups we have that $z_0^{-1}\theta:\cF\to \cF'\oset{Z}{\times}Z'$ is a $Z$-equivariant morphism whose image is in $\cF'\oset{Z}{\times}Z=\cF'$. 
	In other words $z_0^{-1}\theta$ restricts to an isomorphism $\cF\simeq \cF'$ of $Z$-bundles.
	
	In a similar way one can show the injectivity at the level of automorphisms although we will not need it.
\end{proof}

\begin{lemma}\label{L:diag-grps-surj}
	Let $\Gamma'\twoheadrightarrow \Gamma$ be a surjective map of diagonalizable groups. Then 
	\[ \Bun_{\Gamma'}^0(X)\to \Bun_{\Gamma}^0(X) \]
	is surjective on objects.
\end{lemma}
\begin{proof}
(a) First we suppose $\Gamma',\Gamma$ to be tori. 
Since $\Bun^0_{\bG_m}(X)\simeq \Pic^0(X)\times \bB\bG_m$ and that abelian varieties are divisible groups we get the desired surjectivity.

(b) For the general case, let $\Gamma'\hookrightarrow T'$ be an embedding into a torus. Define $T:=(T'\times \Gamma)/\Gamma'$, i.e. the pushout of $T'$ and $\Gamma$ along $\Gamma'$:
\begin{equation}\label{E:diag-diag-gr-surj}
\begin{tikzcd}        
	\Gamma' \arrow[hook]{d}\arrow[twoheadrightarrow]{r} & \Gamma\arrow[hook]{d}\\
	T'\arrow[r,twoheadrightarrow]{}{q} & T   
	\end{tikzcd}
\end{equation}
	
Since the category of diagonalizable groups is abelian (being anti-equivalent to the category of finitely generated abelian groups) this square is also cartesian\footnote{One could also just check it by hand easily.}, i.e. $\Gamma'=q\inv (\Gamma)$.

From \eqref{E:diag-diag-gr-surj} we get a commutative diagram
\begin{equation}\label{E:}
\begin{tikzcd}
\Bun_{\Gamma'}^0(X)\rar \dar[hook] & \Bun_{\Gamma}^0(X) \dar[hook]\\
\Bun_{T'}^0(X) \rar[twoheadrightarrow]{q} & \Bun_T^0(X)
\end{tikzcd}
\end{equation}
where the vertical maps are injective due to Lemma~\ref{L:inj-diaggrps-inject}. The surjectivity of $q$ follows from (a) above. 

Let $\cF'\in \Bun_{T'}^0(X)$ be such that $q(\cF')$ has a reduction to $\Gamma$, i.e. the bundle $q(\cF')/\Gamma$ has a section. 
However, since the diagram \eqref{E:diag-diag-gr-surj} is cartesian we have $q(\cF')/\Gamma=\cF'/\Gamma'$. 
But having a section of $\cF'/\Gamma'$ is equivalent to giving a reduction of $\cF'$ to $\Gamma'$. 
The surjectivity of $q$ implies the desired surjectivity.
\end{proof}

\subsection{The action of the center} 
In this subsection we work over an elliptic curve. We'll analyze in detail the stabilizer of $\cM_{Z(L)}^0$ acting on $\cM_L^{\vlam_L}$ for $L,\vlam_L$ such that there exist stable bundles on the elliptic curve (see Theorem~\ref{T:thmGalois covering}(1)). 

\begin{lemma}
	\label{L:stab M_Z acts on M_L}
	Let $L,\vlam_L$ be as in Theorem~\ref{T:JH for Gbdles} (5). Then the stabilizer of $\cM_{Z(L)}^0(\bk)$ acting on $\cM_L^{\vlam_L}(\bk)$ is precisely $\cM_{Z([L,L])}(\bk)$.
\end{lemma}
\begin{proof}
	Let us put $Z:=Z(L)$ and $\coZ:=L/[L,L]$  the center and the cocenter of $L$.
	
	First we show that if $\cL\in \cM_{Z(L)}^0(\bk)$ stabilises some $\cF\in \cM_L^{\vlam_L}(\bk)$ then $\cL\in \cM_{Z([L,L])}(\bk)$. This is actually quite simple and follows from the commutativity of the diagram	
		\[ \begin{tikzcd}
		L\times Z(L)\dar{=} \rar{\det\times\det} & \coZ\times \coZ \rar{m} & \coZ \dar{=}\\
		L\times Z(L)\rar{m} & L\rar{\det} & \coZ.
		\end{tikzcd}
		 \]
	Indeed, from the diagram we infer that  $\det(\cL\otimes \cF)\simeq \det(\cL)\otimes\det(\cF)$ and hence if $\cL\otimes \cF\simeq \cF$ we obtain $\det(\cL)\simeq \cO$, in other words $\cL$ admits a reduction to $Z([L,L])$.
	
	Notice that here we haven't used the semistability or genus one.
	
	The converse is a bit more technical and uses stability and genus one. Let $\cL$ be a $Z([L,L])$-bundle on $E$ and $\cF$ a stable $L$-bundle on $E$ of degree $\vlL$. 
	We need to show that $\cL\otimes\cF_L\simeq \cF_L$.
	
	Let us split the argument depending on whether $[L,L]$ is simply connected or not.
	
	(a) $[L,L]$ simply connected.  Lemma~\ref{L:embed [L,L] s.c.} provides an embedding $L\subset \prod_i \GL_{n_i}\times T'=:H$ where $T'$ is a torus and such that $[L,L]=[H,H]$. Using Lemma~\ref{L:iso descends to reductions} we can suppose $L=H$ in which case the statement is equivalent to Theorem~\ref{T:At-coprime-unique-det} (2).

	(b) $[L,L]$ arbitrary.  From Lemma~\ref{L:proj to L when [L,L] non s.c.} there is a central extension 
	\[ 1\to T'\to L'\to L\to 1 \]
	 with $T'$ a torus and $[L',L']$ simply connected. 
	
	Pick $\vlam_{L'}\in\pi_1(L')$ a lift of $\vlam_L$. From Theorem~\ref{T:BH-Bun_Z-torsor} we have that the map
	\[\Bun_{L'}^{\vlam_{L'},\st}\to\Bun_L^{\vlam_L,\st}\]
	is a $\Bun_{T'}^0$-torsor, in particular there exists a stable $L'$-bundle $\cF'$ which lifts $\cF$. 
	
	Since $Z([L',L'])$ surjects onto $Z([L,L])$,  Lemma~\ref{L:diag-grps-surj} shows that there exists a $Z([L',L'])$-bundle $\cL'$ which lifts $\cL$. 
	
	Applying (a) to $L'$ we get $\cF'\otimes \cL'\simeq \cF'$ and pushing forward to $L$-bundles and using the commutativity of the following diagram 
		\[ \begin{tikzcd}
		L'\times Z(L') \rar{m'} \dar & L' \dar\\
		L\times Z(L) \rar{m} & L
		\end{tikzcd} \]
	we eventually get $\cF\otimes\cL\simeq \cF$ which concludes the proof.
\end{proof}

\subsection{The determinant map}
\begin{lemma}\label{L:endo stable vbd}
	Let $\cV$ be a stable vector bundle over a smooth projective curve $X$. Then $\End(\cV)=\bk\cdot\Id$.
\end{lemma}
\begin{proof}
	This is well-known and is a version of Schur's Lemma. Let $\phi\in\End(\cV)$ and let $\lam\in\bk$ be an eigenvalue of $\phi$ at some point. Then the endomorphism $\phi-\lam\Id$ of $\cV$ has a non-trivial kernel which must be of degree zero. The stability of $\cV$ implies at once that the kernel must be $\cV$, in other words $\phi=\lambda\Id$.
\end{proof}

\begin{cor}\label{C:endom are center}
	Let $L$ be a reductive group isomorphic to a product of groups of type $A$ and let $\cF$ be a stable $L$-bundle on $E$. Then $\rH^0(E,\fl_\cF)=z(\fl)$.
\end{cor}
\begin{proof}
We have an exact sequence
\[ 0\to z(\fl)=\rH^0(E,z(\fl)_\cF)\to \rH^0(E,\fl_\cF)\to \rH^0(E,\fl^\ad_\cF) \]
and it is enough to show the vanishing $\rH^0(E,\fl^\ad_\cF)=0$. 

If $L=\prod_i \GL_{n_i}$ then Lemma~\ref{L:endo stable vbd} proves  $z(\fl)=\rH^0(E,z(\fl)_\cF)\simeq \rH^0(E,\fl_\cF)$. 
Moreover, Corollary~\ref{C:M_PGLn=pt} implies 
\begin{equation}\label{E:vanish proj endom}
\rH^0(E,\fl^\ad_\cF)=0
\end{equation}

In general, the group $L^\ad\simeq \prod_i \PGL_{n_i}$ is also the adjoint group of $L':=\prod_i \GL_{n_i}$ and the $L^\ad$-bundle $\cF/Z(L)$ lifts to an $L'$-bundle.
Since the bundles $\fl^\ad_\cF$ and $\fl'^\ad_{\cF'}$ are isomorphic, from \ref{E:vanish proj endom} applied to $L'$ we deduce $ \rH^0(X,\fl^\ad_\cF)=0$ which is what we needed.
\end{proof}

\begin{rem}
	The statement of the corollary is true for any reductive group $G$ and any stable $G$-bundle on a smooth projective curve $X$ in characteristic zero. However, the proof is more involved (see for example \cite[Proposition 3.2]{Ram-stable}). 
	I don't know whether it is true in positive characteristic for any smooth projective curve and any group. 
\end{rem}

\begin{lemma} \label{L:G surjects onto H then smooth}
	Let $G\twoheadrightarrow H$ be a surjective map of algebraic groups. Then the map $\Bun_G(X)\to\Bun_H(X)$ is smooth.
\end{lemma}
\begin{proof}
		Denote by $\pi:G\twoheadrightarrow H$ the morphism and let us consider the exact triangle of tangent complexes for the induced map $\pi:\Bun_G\to \Bun_H$
\[ \bT_{\pi} \to \bT_{\Bun_G} \to \pi^*\bT_{\Bun_H} \stackrel{+}{\to}. \]
We obtain a long exact sequence in cohomology whose end terms are
\[\rH^1(X,\fg_\cF)\to \rH^1(X,\fh_\cF)\to \rH^2(X,\fk_\cF)=0,\]
where we put $\fk=\mathrm{Lie}(\ker(\pi))$. 
	It follows that $\pi$ is smooth.
\end{proof}

\begin{lemma} \label{L:det is etale} Let $L,\vlam_L$ be as in Theorem~\ref{T:JH for Gbdles} (5).
	Then $\det:\cM_L^{\vlam_L}\to\cM_{L/[L,L]}^{\det(\vlam_L)}$ is étale.
\end{lemma}
\begin{proof} 
	We need to show that for every stable $L$-bundle $\cF$ the tangent map 
	\[ d_\cF\det:\rH^1(E,\fl_\cF)\to \rH^1(E,(\fl/[\fl,\fl])_{\det(\cF)})  \]
	is an isomorphism. 
	Surjectivity is immediate from Lemma~\ref{L:G surjects onto H then smooth}, in other words $\det$ is smooth.
		
	Let us now compute the dimensions. Since $L/[L,L]$ is a torus we know that $\cM_{L/[L,L]}^\vlam$ is a product of $\dim(L/[L,L])$ connected components of $\Pic(E)$, in particular it is smooth of dimension equal to $\dim(L/[L,L])$.
	
	On the other hand, let $\cF\in \cM_L^{\vlam_L}$. Recall that over an elliptic curve we have $\dim \rH^0(E,\fl_\cF)=\dim \rH^1(E,\fl_\cF)$ by Riemann--Roch. 
	From Corollary~\ref{C:endom are center} we obtain  $\dim(\rH^1(E,\fl_\cF))=\dim(z(\fl))$. Since $\dim(z(\fl))=\dim(\fl/[\fl,\fl])$ for any reductive Lie algebra $\fl$, we're done.
\end{proof}
\subsection{Proof of \cref{T:det is iso}}
	Now we can finish the proof of \cref{T:det is iso}.
	
	We show that
	\[\det:\cM_L^{\vlam_L}\to \cM_{L/[L,L]}^{\det(\vlam_L)}\]
	is finite and bijective on $\bk$-points. We conclude that it is an isomorphism since a finite, étale map of degree one is an isomorphism.
	
	The moduli space $\cM_L^{\vlam_L}$ is a proper variety, in particular $\det$ is a proper map. Since it is also étale (from \cref{L:det is etale}) we deduce it is finite.
	
	Now the surjectivity of $\det:\cM_L^{\vlam_L}(\bk)\to \cM_{L/[L,L]}^{\det(\vlam_L)}(\bk)$ follows at once.
	
	For the injectivity, let $\cF,\cF'\in \cM_L^\vlam(\bk)$ be such that $\det(\cF)\simeq \det(\cF')$. The action of $\cM_{Z(L)}^0(\bk)$ on $\cM_L^\vlam(\bk)$ is transitive by Corollary~\ref{C:M_Z acts trans on M_L}, hence there exists $\cL\in \cM_{Z(L)}^0(\bk)$ such that $\cF'\simeq \cF\otimes \cL$. By taking determinants we have (see proof of Lemma~\ref{L:stab M_Z acts on M_L})
	\[ \det(\cF') \simeq \det(\cF)\otimes\det(\cL).\]
	From the assumption on $\cF$ and $\cF'$ we obtain $\det(\cL)=\cO$, or in other words $\cL$ admits a reduction to $Z([L,L])$.  Lemma~\ref{L:stab M_Z acts on M_L} implies $\cF\otimes \cL\simeq \cF$ and hence $\cF\simeq\cF'$, or in other words $\det$ is injective on $\bk$-points. \qedhere

\begin{rem}\label{R:univ bdle on M_L} 
	Given this simple description of $\cM_L^{\vlam_L}$ one might be led to think that the existence of a universal bundle on it is automatic from the classical Poincar\'e bundle on the Picard variety. This is not the case. For example, if $Z(L)$ is not connected then Theorem 6.8 from \cite{BisHoff-Poincare_families} says that $\cM_L^{\vlam_L}$ does not admit a universal bundle (called Poincar\'e bundle in loc.cit.). On the other hand, the same theorem tells us that if $[L,L]$ is simply connected and $Z(L)$ is connected then there is a universal bundle. I haven't determined precisely what happens if $[L,L]$ is not simply connected, one of the issues being that the automorphism group of a stable $L$-bundle is bigger than $Z(L)$ in this situation.
\end{rem}

\appendix
\section{}

In this Appendix we provide a table (taken from \cite{Fratila2016}) with the Levi subgroups $L_{\vlam_G}$ appearing in Theorem~\ref{T:thmGalois covering}, as well as their relative Weyl groups $W_{L,G}$. We omit $\vlam_G=0$ since in this case the Levi subgroup is always equal to the maximal torus.

\bigskip
{\small 
\begin{center}
\begin{tabular}[b]{|c|c|c|c|c|}
\hline
$G$ & $\vlG$ & \begin{tabular}{c}Type of $L$\end{tabular} & \begin{tabular}{c}Diagram \\ of $(G,L)$\end{tabular} & \begin{tabular}{c} Type of\\ $W_{L,G}$ \end{tabular}\\
\hline
\hline
$A_{n-1}$ & $d$ & \begin{tabular}{c}$\Ane\times\dots\times \Ane$\\ $e=\gcd(n,d)$\end{tabular} & $\boxed{\Ane}\!\!-\!\!\circ\dots\circ\!\!-\!\boxed{\Ane}$ & $A_{e-1}$\\
\hline
\hline
$B_n$ & 1 & ${A_1}^{\vphantom{\sum}}_{\vphantom{\sum}}$ & $\dynkin{B}{oooo.oo*}$ & $C_{n-1}$\\
\hline
\hline
$C_{2n}$ & 1 & $\underbrace{A_1\times A_1\cdots \times A_1}_{n}^{\vphantom{\sum}}$ & 
$\dynkin{C}{*o*o.o*o}$ & $C_n$ \\
\hline
$C_{2n+1}$ & 1 & $\underbrace{A_1\times A_1\dots \times A_1}_{n+1}^{\vphantom{\sum}}$ & $\dynkin{C}{*o*o.*o*}$ & $C_n$\\
\hline
\hline
\multirow{4}{*}{$D_{2n+1}$} & 1 &$A_1\times \dots\times A_1\times A_3$  & 
\dynkin{D}{*o*o.o***} & $C_{n-1}$ \\\cline{2-5}

& 2 &$A_1\times A_1$  & \dynkin{D}{oooo.oo**}
 & $C_{n-1}$\\

\hline
\multirow{6}{*}{$D_{2n}$} & (1,0) &$A_1\times\dots\times A_1$ & \dynkin{D}{*o*o.*o*o}
 & \begin{tabular}{c}$B_n$
\end{tabular}	\\\cline{2-5}
&(0,1) & $A_1\times A_1$ & 

\dynkin{D}{oooo.oo**}
 & \begin{tabular}{c}$C_{2n-2}$
\end{tabular}	\\\cline{2-5}
 & (1,1) & $A_1\times\dots\times A_1$ & 

\dynkin{D}{*o*o.*oo*}
 & $C_n$	\\
\hline
\hline
$E_6$ & $1$ &$A_2\times A_2$ & 

\dynkin{E}{*o*o**}
 & $G_2$ \\
\hline
\hline
$E_7$ & $1$ & $A_1\times A_1\times A_1$ & 

\dynkin{E}{o*oo*o*}
&$F_4$\\
\hline
\end{tabular}
\end{center}}

\end{document}